\documentclass[12pt,letterpaper]{amsart}
\usepackage{amssymb,amsfonts, amsmath, color,graphicx, comment, mathtools, amsthm}

\usepackage[
backend=biber,
style= alphabetic,
sorting=nyt
]{biblatex}
\usepackage[backend=biber]{biblatex}
\renewbibmacro{in:}{}
\usepackage[hidelinks]{hyperref}
\usepackage{tikz-cd}
\usepackage[all]{xy}
\theoremstyle{plain}
\newtheorem{main}{Theorem}

\newtheorem{thm}{Theorem}[section]
\newtheorem*{thm*}{Theorem}
\newtheorem{lem}[thm]{Lemma}

\newtheorem{prop}[thm]{Proposition}
\newtheorem*{prop*}{Proposition}

\newtheorem{cor}[thm]{Corollary}
\newtheorem*{cor*}{Corollary}

\theoremstyle{definition}
\newtheorem{defn}[thm]{Definition}
\newtheorem*{defn*}{Definition}

\newtheorem*{question*}{Question}

\newtheorem*{conv*}{Convention}

\newcommand{\N}{\mathbb{N}}

\renewcommand{\phi}{\varphi}

\usepackage{tcolorbox}

\def\N{{\mathbb{N}}}

\def\u{{\widehat{\mathcal{U}}}}
\def\a{{\alpha}}
\def\b{{\beta}}

\title{Tensorial Permanence of $K$-Stability for Diagonal AH-Algebras}
\author{Apurva Seth}
\address{Apurva~Seth,
Mathematical Institute, University of Oxford, Radcliffe Observatory, Andrew Wiles Building, Woodstock Rd, Oxford OX2 6GG, United Kingdom.}
\email{apurva.seth@maths.ox.ac.uk, apurvaseth14@gmail.com}

\subjclass[2020]%
{Primary
46L85;  Secondary 46L80.
}
\begin{document} 

\begin{abstract}
We study $K$-stability for tensor products of diagonal AH-algebras with arbitrary C*-algebras. Our main result provides a characterization of $K$-stability: For a diagonal AH-algebra $A = \varinjlim (A_i, \varphi_i)$, $A \otimes B$ is $K$-stable for every C*-algebra $B$ if and only if the sizes of the matrix blocks in the inductive system grow without bound. As applications, we show that non-$\mathcal{Z}$-stable Villadsen algebras of the first kind are $K$-stable when tensored with any C*-algebra. Moreover, any simple, unital, infinite-dimensional diagonal AH-algebra automatically satisfies this growth condition, and therefore its tensor product with arbitrary C*-algebras is always $K$-stable.
\end{abstract}
\maketitle

\section{Introduction}
Nonstable $K$-theory for C*-algebras is the study of homotopy groups of the unitary group of a C*-algebra. These groups generally contain finer topological information than ordinary $K$-theory, but they are notoriously difficult to compute, even for the simplest examples. For instance, when $A=\mathbb{C}$, $\pi_1(U(\mathbb{C}))\cong \mathbb{Z}$, and  $\pi_n(U(\mathbb{C}))=0$, for $n>1$. For higher matrix algebras $A=M_k(\mathbb{C})$, the unitary group has much more intricate topology, and $\pi_n(U_k(\mathbb{C}))$ remains unknown beyond small dimensions. 
Rieffel~(\cite{rieffel12}) showed that for the noncommutative irrational torus, the homotopy groups of the unitary group are naturally isomorphic to the corresponding 
$K$-theory groups. At that time, this phenomenon was not given a formal name. Thomsen~(\cite{thomsen}) later extended the theory to incorporate nonunital C*-algebras and formalized the property, introducing the term \emph{$K$-stability} to describe C*-algebras for which the homotopy groups of the unitary group coincide with the corresponding 
$K$-theory groups.
In such cases, the computation of nonstable $K$-theory reduces to that of ordinary $K$-theory, which is often easier to handle and better understood, and many interesting C*-algebras have been shown to be $K$-stable. This class includes 
\begin{itemize}
    \item noncommutative irrational torus (\cite{rieffel12}),
    \item purely infinite simple C*-algebras (\cite{zhang2001}),
    \item nonelementary simple C*-algebras with real rank zero and stable rank one (\cite{zhang}). 
\end{itemize}

Given that 
$K$-stability appears in such diverse classes of C*-algebras, it is natural to ask how this property behaves under various C*-algebraic constructions. Indeed, 
$K$-stability is preserved under matrix amplifications, finite direct sums, inductive limits, and extensions (in the sense that if two out of three algebras in a short exact sequence are 
$K$-stable, so is the third). These permanence properties lead naturally to the question of how 
$K$-stability behaves under taking tensor products.
Formally, the question is:
if $A$ is $K$-stable, is the minimal tensor product $A\otimes B$ $K$-stable for all C*-algebras $B$, or for all $B$ in  some class of C*-algebras?
It is important to note that we ask whether
$K$-stability of a single factor is strong enough to ensure that the tensor product remains 
$K$-stable. In many familiar situations, this is indeed the case. For instance, $C(X)$ is not 
$K$-stable for any compact Hausdorff space $X$, yet $C(X)\otimes A$ is 
$K$-stable whenever $A$ is (\cite[Lemma 2.1]{vaidyanathan}). Likewise, although $M_n(\mathbb{C})$ is not 
$K$-stable for any $n\in\mathbb{N}$, the tensor product $M_n(\mathbb{C})\otimes A$ is 
$K$-stable whenever $A$ is.

Despite these examples, this kind of tensorial permanence is far from automatic. Tensoring can alter the topology of the unitary group in subtle ways. In general, there is no canonical way to identify unitaries in the tensor product of two C*-algebras with unitaries coming from either factor.
 However, for many classes of C*-algebras, $K$-stability is preserved under taking tensor products.  For instance,  Thomsen~(\cite{thomsen}) showed that if $A$ is either the Cuntz algebra $\mathcal{O}_n$ or an infinite-dimensional simple AF-algebra, then $A \otimes B$ is $K$-stable for every C*-algebra $B$. Furthermore, if $A$ is $\mathcal{Z}$-stable, then by the results of Jiang and Su~(\cite{jiang1997}) and Hua~(\cite{hua2025k}), $A$ is $K$-stable, and so is $A\otimes B$ for any C*-algebra $B$.  

To better understand these questions, we turn to diagonal AH-algebras, which are inductive limits of finite direct sums of homogeneous 
C*-algebras with connecting maps that are diagonal in a suitable sense (see Definition~\ref{def_diagonal_maps}). Diagonal AH‑algebras have been studied for several decades and have played a significant role in the theory of C*-algebras. Historically, they were used to construct examples of AH-algebras that did not have slow dimension growth (\cite{villadsen1998simple}), showing that inductive limits of homogeneous C*-algebras can behave in unexpected ways. Subsequently, they were employed to provide counterexamples to Elliott’s classification conjecture (\cite{Andrew_classification}). These developments highlight the relevance of diagonal AH-algebras as a testing ground for structural phenomena.

In this paper, we study the 
$K$-stability of tensor products of these algebras. The fact that the connecting maps in these inductive systems are diagonal is used in an essential way to give a complete characterization of $K$-stability of the tensor products of such algebras. For general AH-algebras, the situation is much more subtle. In fact, the problem is already difficult in the simple case: Villadsen~(\cite{MR1916650}) constructed examples of simple AH-algebras that are not 
$K$-stable, showing the limitations in the non-diagonal setting. The main result of the paper is the following.
\begin{main}[Theorem~\ref{mainthm}]\label{mainthm1}
Let 
\[
A = \varinjlim (A_i, \varphi_i)
\] 
be a diagonal AH-algebra. Then the following statements are equivalent:
\begin{enumerate}
    \item $\lim\limits_{i \to \infty} d_{A_i} = \infty$.
    \item $A \otimes B$ is $K$-stable for any C*-algebra $B$.
\end{enumerate}
\end{main}

The notion $d_A$ for a homogeneous C*-algebra $A$ is the minimal matrix size appearing among the homogeneous summands of $A$; see Section~\ref{sec_main_result} for the precise definition. As an application, it follows that if $A$ is a unital $K$-stable AF-algebra or a non-$\mathcal{Z}$-stable Villadsen algebra of the first type, then $A\otimes B$ is $K$-stable for any C*-algebra $B$.

Rieffel (\cite{rieffel12}) showed that for a unital C*-algebra with stable rank one, the natural inclusions of unitary groups induce isomorphisms \[\pi_0(U_n(A))\cong \pi_0(U_{n+1}(A))\cong K_1(A),\] for all $n\geq 1$, giving partial $K$-stability results. Elliott, Ho, and Toms~(\cite{elliott2009class}) have further shown that simple, unital, diagonal AH-algebras always have stable rank one. Using Theorem~\ref{mainthm1}, we can extend this partial picture: If such an algebra is also infinite-dimensional, it automatically satisfies the growth condition in Theorem~\ref{mainthm1}, and consequently,  tensor products of these algebras with arbitrary C*-algebras are 
$K$-stable.

\subsection*{Acknowledgements} I would like to express my heartfelt thanks to Stuart White for his guidance and insightful ideas throughout this project. I am also very grateful to Andrew Toms for many discussions and valuable remarks. I would further like to thank Jakub Curda and Shanshan Hua for their thoughtful and encouraging comments on the first draft of this paper.

 \subsection*{Funding}
 The author is supported by the Engineering and Physical Sciences 
Research Council (EP/X026647/1). For the purpose of open access, the author 
has
applied a CC BY public copyright license to any author-accepted manuscript 
arising from this submission.

\section{Preliminaries}\label{sec_pre}
In this section, we recall the definitions and basic facts concerning diagonal AH-algebras and Thomsen’s construction of nonstable $K$-theory groups for C*-algebras.
\subsection{AH systems with diagonal maps}

Let us recall the notion of \emph{diagonal AH-algebras} as introduced in \cite{elliott2009class}.  Throughout the paper, for $m\in\mathbb{N}$, $M_m$ will always mean the finite-dimensional C*-algebra $M_m(\mathbb{C})$, and we will use these notations interchangeably.  

Let $X$ and $Y$ be compact Hausdorff spaces. A unital $*$-homomorphism 
\[
\varphi\colon C(X)\otimes M_m\rightarrow C(Y)\otimes M_{km}
\]
is called $\emph{diagonal}$, if there exist continuous maps $\lambda_1,\lambda_2,\hdots,\lambda_k\colon Y\rightarrow X$ such that 
\[
\varphi(f)=
\begin{pmatrix}
  f\circ \lambda_1 &0  &\dots &0 \\
  0               & f\circ\lambda_2 &\dots &0 \\
  \vdots &\vdots &\ddots &\vdots \\
  0 &0 &\dots & f\circ \lambda_k 
\end{pmatrix}.
\]  

In this case, one can associate $\phi$ with a  matrix $\begin{pmatrix}( \lambda_1,\lambda_2,\dots,\lambda_k)\end{pmatrix}$, whose entry is a tuple consisting of the functions $\{\lambda_i\}_{i=1}^k$, called as the \emph{eigenvalue functions} for $\varphi$. 

\begin{defn}\label{def_diagonal_maps}
    A unital $*$-homomorphism 
    \[
\phi\colon \bigoplus_{i=1}^n C(X_i)\otimes M_{m_i}\rightarrow C(Y)\otimes M_p
    \]
    is \emph{diagonal} if there exist natural numbers $k_1,k_2,\dots, k_n$ such that $\sum_{i=1}^n k_i m_i=  p$, and 
    \[
    \prescript{i}{}{}\phi \colon C(X_i)\otimes M_{m_i}\to C(Y)\otimes M_{k_im_i}
    \]
    is a diagonal map as defined earlier with matrix $\begin{pmatrix}(\lambda_1^i,\lambda_2^i,\dots \lambda_{k_i}^i)\end{pmatrix}$ such that 
    \[
    \phi=\text{diag}(\prescript{1}{}{}\phi,\prescript{2}{}{}\phi,\dots,\prescript{n}{}{}\phi).
    \]
\end{defn}

Note that, in the above definition, $k_i=0$ is allowed.   In this way, we can associate to $\phi$, an $n\times 1$ matrix whose entries are tuples consisting of the eigenvalue functions  associated to $\prescript{i}{}{}\phi$, for $i=1,2,\dots,n$ as given below:
\[
\begin{pmatrix}
    ( \lambda_1^1,\lambda_2^1,\dots, \lambda_{k_1}^1)\\[6pt]
    ( \lambda_1^2,\lambda_2^2,\dots, \lambda_{k_2}^2)\\
    \vdots\\[6pt]
    ( \lambda_1^n,\lambda_2^n,\dots, \lambda_{k_n}^n)
\end{pmatrix}.
\]

Furthermore, we say that a unital $*$-homomorphism 
\[
\phi\colon \bigoplus_{i=1}^n C(X_i)\otimes M_{m_i}\rightarrow \bigoplus_{j=1}^t C(Y_j)\otimes M_{p_j}
\]
is \emph{diagonal} if each restriction
\[
\phi^j\colon\bigoplus_{i=1}^n C(X_i)\otimes M_{m_i}\rightarrow C(Y_j)\otimes M_{p_j}
\]
is diagonal. As above, we can associate to $\phi$ an $n\times t$ matrix 

\[
\begin{pmatrix}
    (\lambda_{1}^{11}, \lambda_2^{11}, \dots, \lambda_{k_{11}}^{11}) & ( \lambda_{1}^{12}, \lambda_2^{12}, \dots, \lambda_{k_{12}}^{12}) & \dots & (\lambda_{1}^{1t}, \lambda_2^{1t}, \dots, \lambda_{k_{1t}}^{1t}) \\[6pt]
    (\lambda_{1}^{21}, \lambda_2^{21}, \dots, \lambda_{k_{21}}^{21}) & (\lambda_{1}^{22}, \lambda_2^{22}, \dots, \lambda_{k_{22}}^{22}) & \dots & (\lambda_{1}^{2t}, \lambda_2^{2t}, \dots, \lambda_{k_{2t}}^{2t}) \\[6pt]
    \vdots & \vdots & \ddots & \vdots \\[6pt]
    (\lambda_{1}^{n1}, \lambda_2^{n1}, \dots, \lambda_{k_{n1}}^{n1}) & (\lambda_{1}^{n2}, \lambda_2^{n2}, \dots, \lambda_{k_{n2}}^{n2}) & \dots & (\lambda_{1}^{nt}, \lambda_2^{nt}, \dots, \lambda_{k_{nt}}^{nt})
\end{pmatrix},
\]
where the $(i,j)$th entry is the tuple $( \lambda_{1}^{ij}, \lambda_2^{ij},\dots,\lambda_{k_{ij}}^{ij})$ associated to the diagonal map
\[
\prescript{i}{}{\phi}^{\,j}\colon C(X_i)\otimes M_{m_i}\rightarrow C(Y_j)\otimes M_{k_{ij}m_i}.
\]

An \emph{AH system with diagonal maps} is an inductive system 
\[
(A_i,\phi_i), \qquad i\in\mathbb{N},
\]
where each algebra
\begin{equation}\label{eqn:homogenous_summand}
A_i=\bigoplus_{l=1}^{n_i} C(X_{i,l})\otimes M_{m_{i,l}}
\end{equation}
is a finite direct sum of homogeneous building blocks with 
$X_{i,l}$ a connected compact metric space and $m_{i,l},n_i\in\mathbb{N}$, 
and where each connecting map $\phi_i$ is diagonal.
A C*-algebra $A$ is called a \emph{diagonal AH-algebra} if it is the inductive limit of some AH system with diagonal maps.  Any such inductive system $(A_i,\phi_i)$ will be called a \emph{decomposition} of $A$, and we write
\[
A = \varinjlim (A_i,\phi_i).
\]
 Note that the composition of a diagonal $*$-homomorphism is again diagonal.  Furthermore, it has been shown in \cite{elliott2009class} that the connecting maps $\varphi_i$ can be assumed to be injective; hence, we assume from this point onward that all connecting maps in diagonal AH systems are injective.
 
 \medskip
 
Lastly, let $A=\varinjlim (A_i,\phi_i)$ be a diagonal AH-algebra with $A_i$ as in~\eqref{eqn:homogenous_summand}, and $B$ be any C*-algebra (not necessarily unital). Then $A \otimes B$ is the inductive limit of 
\[
\big(A_i \otimes B, \ \phi_i \otimes \mathrm{id}_B\big),
\] 
where the maps $\phi_i \otimes \mathrm{id}_B$ are defined as follows: For $1\leq l\leq n_i$  and $1\leq j\leq n_{i+1}$,
     define  
    \[
   \prescript{l}{}(\phi_i \otimes \mathrm{id}_B)^{\,j} \colon C(X_{i,l}) \otimes M_{m_{i,l}}(B) \longrightarrow C(X_{i+1,j}) \otimes M_{k_{lj} m_{i,l}}(B)
    \] 
    as \[f\mapsto \text{diag}(f\circ \lambda_1^{lj}, \ldots, f\circ \lambda_{k_{lj}}^{lj} ),\]
      where
    \[
    \lambda_1^{lj}, \dots, \lambda_{k_{lj}}^{lj} \colon X_{i+1,j} \to X_{i,l}
    \] 
  are eigen value functions associated to $\prescript{l}{}{\phi_i}^{j}$. 
With this convention, $A \otimes B$ is canonically expressed as an inductive limit reflecting the original diagonal AH decomposition of $A$, and the maps $\phi_i \otimes \mathrm{id}_B$ are, in a natural sense, \emph{diagonal with respect to $B$}. In particular, this shows that matrix algebras over $A$ again admit an AH decomposition with diagonal connecting maps, such that the matrix representation of the connecting maps remains unchanged. Throughout the paper, we shall work with such decomposition for tensor products of diagonal AH-algebras.

\subsection{Nonstable $K$-theory}
Let us now recall the work of Thomsen
of constructing the nonstable $K$-groups associated to a C*-algebra. For more details, the reader may refer to \cite{thomsen}. 

Let $A$ be a C*-algebra (not necessarily unital). Define an associative composition $\bullet$ on $A$ by
\begin{equation}\label{eqn:composition}
a\bullet b=a+b-ab.
\end{equation}
An element $u\in A$ is said to be a quasi-unitary if
\[
u\bullet u^{\ast} = u^{\ast}\bullet u = 0.
\]
We write $\u(A)$ for the set of all quasi-unitary elements in $A$. Then, $\u(A)$ is a topological group with respect to the binary operation $\bullet$, with the identity element being $0_A$. If $B$ is a unital C*-algebra, we write $U(B)$ for its group of unitaries. Let $A^+$ denote the minimal unitization
of $A$. It follows by \cite[Lemma 1.2]{thomsen} that an element $a\in A$ is a quasi-unitary
if and only if $(1-a)\in U(A^+)$.

Let $X$ be any topological space and $G$ be a topological group. Choose a base point $+\in X$.
We can then consider the group $[X,G]_0$ consisting of the homotopy classes of base
point preserving continuous maps from $X$ to $G$, where the base point of $G$ is the identity element $e\in G$. Given a continuous base point preserving map
$f\colon X\rightarrow G$, we let $[f]$ denote its class in $[X,G]_0$.
Then, the group structure of $G$ gives $[X,G]_0$ the structure of a group, and  that $[X,\u(\cdot)]_0$ is a covariant  and homotopy invariant half-exact
functor from the category of C*-algebras to the category of groups.

Let $SA$ denote the suspension of the C*-algebra $A$, that is, 
\[
SA:=\left\{f\colon[0,1]\rightarrow A,\,\,\text{such that}\,\, f(0)=f(1)=0\right\}\cong C_*(S^1,A),
\]
where $C_*(S^1,A):=\{f\colon S^1\to A\,\,\, \text{such that}\,\, f(1)=0\}$.

For $n>1$, we set $S^nA= S(S^{n-1}(A))$.
The following result can be found in \cite[Lemma 2.3]{thomsen}.

\begin{lem}
    $[S^n, \u(A)]_0\cong \pi_n(\u(A))\cong \pi_0(\u(S^n A))$ for $n=0,1,2,\dots$.
\end{lem}

Moreover, the above isomorphism is natural in $A$.
\begin{defn}
    For $m \geq 0$, the \emph{nonstable $K$-theory groups} associated to a C*-algebra $A$ are defined to be
    \[
    k_m(A) := \pi_m(\u(A)).
    \]
\end{defn}
As a consequence of \cite[Proposition 2.1]{thomsen}  and \cite[Theorem 4.4]{handelman1978k}, we obtain that $k_m$ is a continuous homology theory.
\begin{defn}\label{Def_K_stable}
Let $A$ be a C*-algebra and $n\geq 2$. Let $\iota_{n,A}\colon M_{n-1}(A)\rightarrow M_{n}(A) $ be  the canonical inclusion map given by
\[
a\mapsto \begin{pmatrix}
    a&0\\
    0&0
\end{pmatrix}.
\]
Then $A$ is said to be \emph{$K$-stable} if the induced map $k_m(\iota_{n,A})\colon k_m(M_{n-1}(A))\rightarrow k_m(M_n(A))$ is an isomorphism for all $m\geq 0$ and $n\geq 2$.
\end{defn}
\section{Main Result}\label{sec_main_result}
In this section, we prove the main result of the paper and give some applications. We begin by fixing some notation that will be used throughout the section.
\subsection{Notational conventions}
 We fix some notational conventions we will use repeatedly: If $A$ is a C*-algebra and $a,b\in\u(A)$, we shall write $a\sim_h b$ in $\u(A)$ whenever there is a continuous map $f\colon[0,1]\rightarrow \u(A)$ such that $f(0)=a$ and $f(1)=b$. We shall denote $[a]$ to be the homotopy class of $a\in\u(A)$, and write $[a]=0$ in $\pi_0(\u(A))$, whenever $a\sim_h 0_A$ in $\u(A)$. Given $A=\bigoplus_{i=1}^n C(X_i)\otimes M_{m_i}$, we shall denote 
 \[
 d_A=\min\{m_1,m_2,\dots, m_{n}\}.
 \]
For C*-algebras $A$ and $B$, we shall denote $C_*(S^n, A)$ to be the C*-algebra of base point preserving functions $f \colon S^n \to A$, where the base point of $A$ is taken to be $0$. We denote by $A \otimes B$ the minimal tensor product of $A$ and $B$.  Furthermore, given an inductive limit decomposition $(A_i, \phi_i)$, we shall denote by 
\[
\phi_{p,r} \colon A_r \longrightarrow A_p
\]
the connecting $*$-homomorphism, defined for $r \leq p$ by 
\[
\phi_{p,r} = \phi_{p-1} \circ \phi_{p-2} \circ \cdots \circ \phi_r.
\]

We state below a few elementary lemmas that will be used in the proof of Theorem~\ref{mainthm}.

\begin{lem}\label{homotopy_lem}
Let $X_1$ and $X_2$ be topological spaces and $G$ be a topological group. Let $\lambda\colon X_2\rightarrow X_1$ be a continuous map. Let $C_{X_i,e}\in C(X_i,G)$, $i=1,2$  be the constant map at the identity $e\in G$. Let $v\sim_h C_{X_1,e}\in C(X_1,G)$. Then, $v\circ \lambda\sim_h C_{X_2,e}\in C(X_2,G)$.   
\end{lem}
\begin{proof}
 Let $H\colon[0,1]\rightarrow C(X_1,G)$ be such that
 \[
 H(0)=v\quad \text{and}\quad H(1)=C_{X_1,e}.
 \]
 Define $F\colon [0,1]\rightarrow C(X_2,G)$ as
 \[
 F(t)=H(t)\circ \lambda.
 \]
 Then, $F(0)=v\circ\lambda$, and $F(1)=C_{X_2,e}$, giving the required homotopy. 
\end{proof}

The following lemma is a version of Whitehead’s lemma (see \cite[Lemma 2.2.5]{rordam2000introduction}) for quasi-unitaries.

\begin{lem}\label{whiteheads}
    Let $B$ be any C*-algebra and $x,y\in \u(B)$. Then
    \[
    \begin{pmatrix}
        x &0\\
        0 &y
    \end{pmatrix}\sim_h \begin{pmatrix}
        x\bullet y &0\\
        0 &0
    \end{pmatrix}\sim_h \begin{pmatrix}
        y&0\\
        0&x
    \end{pmatrix} \,\,\,\,\text{in}\,\,\,\, \u_2(B).
    \]
\end{lem}
\begin{proof}
    Let $H\colon[0,1]\to U(M_2(\mathbb{C}))\subseteq U(M_2(B^+))$ be a continuous path such that 
     $H(0)=\begin{pmatrix}
        1 &0\\
        0 &1
    \end{pmatrix}$ and $H(1)=\begin{pmatrix}
        0 &1\\
        1 &0
    \end{pmatrix}$. Define  $F\colon[0,1]\to U(M_2(B^+))$ as
    \[
    F(t) := H(t)\begin{pmatrix}
        1-x&0\\
        0&1
    \end{pmatrix}H(t)^*.
    \]
It is then straightforward to see that the scalar part of $F(t)$ is 
$
\begin{pmatrix}
1 & 0 \\[2pt]
0 & 1
\end{pmatrix}.
$
Consequently, for each $t\in[0,1]$, the scalar part of  $1_{M_2(B^+)} - F(t)$ is equal to zero, and thus lies in
\[
\bigl(1_{M_2(B^+)} - U(M_2(B^+))\bigr) \cap M_2(B) = \u_2(B).
\]

The map $t \mapsto 1_{M_2(B^+)} - F(t)$ therefore defines a continuous path of quasi-unitaries in $M_2(B)$ connecting
\[
\begin{pmatrix}
x & 0 \\[2pt]
0 & 0
\end{pmatrix}
\quad\text{and}\quad
\begin{pmatrix}
0 & 0 \\[2pt]
0 & x
\end{pmatrix}.
\]

Hence \[\begin{pmatrix}
        x &0\\
        0 &y
    \end{pmatrix}= \begin{pmatrix}
        x &0\\
        0&0
    \end{pmatrix}\bullet \begin{pmatrix}
        0 & 0\\
        0& y
    \end{pmatrix}\sim_h \begin{pmatrix}
         x &0\\
        0&0
    \end{pmatrix}\bullet \begin{pmatrix}
        y &0\\
        0&0
    \end{pmatrix}= \begin{pmatrix}
        x\bullet y &0\\
        0 &0
    \end{pmatrix}.\]
    The second homotopy follows by similar arguments. 
\end{proof}

The next lemma allows us to show $K$-stability for a class of C*-algebras that is closed under matrix amplifications and tensoring with $C(S^m)$, by reducing the problem to showing that $k_0(\iota_2,\cdot)$ is an isomorphism.

 \begin{lem}\label{lem_first}
     Let $B$ be a C*-algebra and let $\mathcal{C}$ be a class of C*-algebras such that for any $A\in \mathcal{C}$, the map
     \[
     k_0(\iota_{2,A\otimes B})\colon\pi_0(\u(A\otimes B))\rightarrow\pi_0(\u_2(A\otimes B))
     \]
     is an isomorphism. If $\mathcal{C}$ is closed under matrix amplifications and taking tensor products with $C(S^m)$ for all $m\in\N$, then $A\otimes B$ is $K$-stable for any  $A\in \mathcal{C}$.
 \end{lem}
 \begin{proof}
     Let $A\in\mathcal{C}$, $m,n\in \N$. We shall show that 
     \[
     k_m(\iota_{{n+1},A\otimes B})\colon \pi_m(\u_n(A\otimes B))\rightarrow \pi_m(\u_{n+1}(A\otimes B))
     \]
     is an isomorphism.

For this, we claim that it suffices to show that
\begin{equation}\label{eq_lem_first}
    k_0(\iota_{n+1,A\otimes B})\colon \pi_0(\u_{n}(A\otimes B))\rightarrow \pi_0(\u_{n+1}(A\otimes B))
    \end{equation}
    is an isomorphism for all $n\geq 1$ and $A\in \mathcal{C}$. Since $C:=C(S^m)\otimes A$ belongs to $\mathcal{C}$, the above claim implies that the map 
    \[
    k_0(\iota_{{{n+1},C\otimes B}})\colon  \pi_0(\u_n(C(S^m)\otimes A\otimes B))\rightarrow \pi_0(\u_{n+1}(C(S^m)\otimes A\otimes B))
    \]
    is an isomorphism for all $m\geq 0$, $n\geq 1$ and $A\in \mathcal{C}$.
     Hence,
     the map 
\[
k_m(\iota_{{n+1},A\otimes B})\colon\pi_m(\u_{n}(A\otimes B)) \rightarrow \pi_m(\u_{n+1}(A\otimes B))
\]
is an isomorphism, following from the naturality of the isomorphism  
\[
\pi_0(\u(C_*(S^m)\otimes M_n(A\otimes B)))\cong \pi_m(\u_n(A\otimes B))
\]
and the commutative diagram below:

\begin{tikzcd}[column sep=0.7em, row sep=1.8em]
0 \arrow{r} & \pi_0(\u(C_*(S^m)\otimes M_{n}(A\otimes B)))\arrow{r}\arrow{d} & \pi_0(\u(C(S^m)\otimes M_{n}(A\otimes B)))\arrow{r}\arrow{d}{k_0(\iota_{n+1,C\otimes B})}&\pi_0(\u_{n}(A\otimes B))\arrow {r} \arrow{d}{k_0(\iota_{n+1,A\otimes B})}&0\\
0\arrow{r} & \pi_0(\u(C_*(S^m)\otimes M_{n+1}(A\otimes B))) \arrow{r} & \pi_0(\u(C(S^m)\otimes M_{n+1}(A\otimes B)))\arrow{r} & \pi_0(\u_{n+1}(A\otimes B))\arrow{r} & 0.
\end{tikzcd}

    What remains to show is the above claim. That is, we now show that \[k_0(\iota_{n+1,A\otimes B})\colon \pi_0(\u_{n}(A\otimes B))\rightarrow \pi_0(\u_{n+1}(A\otimes B))\] is an isomorphism, for any $n\in\N$. 
    For this, see that $M_n(A), M_{n+1}(A)\in \mathcal{C}$. Hence, the maps 
     \begin{equation}
         \begin{split}
             k_0\left({\iota_{2,M_n(A\otimes B)}}\right)\colon \pi_0(\u_n(A\otimes B)) &\xrightarrow{\cong} \pi_0(\u_{2n}(A\otimes B))\,\,\text{and},\\
             k_0\left({\iota_{2,M_{n+1}(A\otimes B)}}\right)\colon \pi_0(\u_{n+1}(A\otimes B)) &\xrightarrow {\cong}\pi_0(\u_{2n+2}(A\otimes B))
         \end{split}
     \end{equation}
     are isomorphisms.
     Hence, we get a  commutative diagram
     \[
\begin{tikzcd}[column sep=4em]
    \pi_0(\u_n(A\otimes B))
        \arrow[r, "k_0(\iota_{{n+1},A\otimes B})"]
        \arrow[rr, bend right=20,
            "\cong",
            "k_0\!\left({\iota_{2,M_n(A\otimes B)}}\right)"' below]
    &
    \pi_0(\u_{n+1}(A\otimes B))
        \arrow[r, "k_0(\iota_{2n,{n+1}})"]
        \arrow[rr, bend right=20,
            "\cong",
            "k_0\!\left({\iota_{2,M_{n+1}(A\otimes B)}}\right)"' below]
    &
    \pi_0(\u_{2n}(A\otimes B))
        \arrow[r, "k_0(\iota_{2n+2,{2n}})"]
    &
    \pi_0(\u_{2n+2}(A\otimes B)),
\end{tikzcd}
\]
where $\iota_{2n,{n+1}}$ is the composition \[\iota_{2n,A\otimes B}\circ\ldots\circ\iota_{n+2, A\otimes B}\colon M_{n+1}(A\otimes B)\rightarrow M_{2n}(A\otimes B),\] and similarly $\iota_{2n+2,2n}$ is the composition \[\iota_{2n+2,A\otimes B}\circ\ldots\circ\iota_{2n+1,A\otimes B}\colon M_{2n}(A\otimes B)\rightarrow M_{2n+2}(A\otimes B).\] It is then easy to see that $k_0(\iota_{{n+1},A\otimes B})$   is an isomorphism. 
\end{proof}

We now give the proof of Theorem~\ref{mainthm1}, which we divide into two parts.  First, we show that if the minimal matrix size appearing in homogeneous summands of the algebras in the AH system of a diagonal AH-algebra $A$ grows without bound, then the tensor product $A\otimes B$ is $K$-stable for any C*-algebra $B$.

\begin{prop}[Growth implies $K$-stability]\label{thm:growth-to-K}
Let 
\[
A = \varinjlim (A_i, \varphi_i)
\] 
be a diagonal AH-algebra. If
\[
\lim_{i\to\infty} d_{A_i} = \infty,
\]
then \(A \otimes B\) is \(K\)-stable for any C*-algebra \(B\).
\end{prop}

\begin{proof}
Let $B$ be a C*-algebra and let $\mathcal{C}$ denote the class of diagonal AH-algebras  $A=\lim\limits_{i\to\infty}(A_i,\phi_i)$, such that  $\lim\limits_{i \to \infty} d_{A_i} = \infty$. Then, for any $A \in \mathcal{C}$, the algebras $C(S^m)\otimes A$ and $M_n(A)$ also belong to $\mathcal{C}$ for all $n,m \in \mathbb{N}$. Thus, by Lemma~\ref{lem_first}, it suffices to show that
    
    \begin{equation}\label{iso_eqn}
   k_0(\iota_{2,A\otimes B})\colon \pi_0(\u(A\otimes B))\rightarrow \pi_0(\u_2(A\otimes B)) 
    \end{equation}
    is an isomorphism.

 We first consider the injectivity of the map in ~\eqref{iso_eqn}. The idea is that the matrix amplification of an element at a finite stage is homotopic to zero may not imply it is homotopic to zero itself. However, by passing to a sufficiently large stage in the inductive system, one can arrange that its image becomes homotopic to zero. The details are as follows.

Since the connecting maps are diagonal, using Lemma~\ref{whiteheads}, we see that the following diagram is commutative.
\[
\xymatrix@C=3em{
\ldots\ar[r] & \pi_0(\u(A_i\otimes B))\ar[rr]^{k_0(\phi_i\otimes\,\text{id}_B)}\ar[d]_{k_0(\iota_{2, A_{i}\otimes B})} && \pi_0(\u(A_{i+1}\otimes B))\ar[d]^{k_0(\iota_{2, A_{i+1}\otimes B})}\ar[r] & \ldots \\
\ldots\ar[r] &  \pi_0(\u_2(A_i\otimes B))\ar[rr]_{k_0(\phi_i\otimes\, \text{id}_B\otimes \text{id}_{M_2})} && \pi_0(\u_2(A_{i+1}\otimes B))\ar[r] & \ldots
}
\]
Hence, by continuity of $k_0$, it suffices to show the following:
Let $[v]\in \pi_0(\u(A_r\otimes B))$ for some $r\in\N$, such that 
\[
k_0(\iota_{2, A_{r}\otimes B})([v])=\bigg[\begin{pmatrix}
    v&0\\
    0&0
\end{pmatrix}\bigg]=0\in \pi_0(\u_2(A_r\otimes B)).
\]
Then there exists some $p\geq r\in\N$, such that \[k_0(\phi_{p,r}\otimes \text{id}_B)([v])=0\in\pi_0(\u(A_p\otimes B)).\]
To this end, let $A_r=\bigoplus_{i=1}^n C(X_i)\otimes M_{m_i}$, and let $[v_1]\oplus [v_2]\oplus\dots\oplus [v_{n}]$ be the image of $[v]$ under the isomorphism 
\[
\pi_0(\u(A_r\otimes B))\cong \bigoplus_{i=1}^n \pi_0(\u(C(X_i)\otimes M_{m_i}(B))).
\]
From the commutative diagram 
\[ \begin{tikzcd}
\pi_0(\u(A_r\otimes B)) \arrow{r}{\cong} \arrow[swap]{d}{\iota_r^*} & \bigoplus_{i=1}^n \pi_0(\u(C(X_i)\otimes M_{m_i}(B))) \arrow{d}{j^*} \\%
\pi_0(\u_{2}(A_r\otimes B)) \arrow{r}{\cong}& \bigoplus_{i=1}^n \pi_0(\u_{2}(C(X_i)\otimes M_{m_i}(B))),
\end{tikzcd}
\]
we get 
\begin{equation}\label{homotopy_eqn}
\bigg[\begin{pmatrix}
    v&0\\
    0&0
\end{pmatrix}\bigg]=0=\bigg[\begin{pmatrix}
    v_1&0\\
    0&0_{m_1}
\end{pmatrix}\bigg]\oplus \bigg[\begin{pmatrix}
    v_2&0\\
    0&0_{m_2}
\end{pmatrix}\bigg]\oplus\dots\oplus\bigg[\begin{pmatrix}
    v_n&0\\
    0&0_{m_n}
\end{pmatrix}\bigg].
\end{equation}
Since, $\lim\limits_{i\to\infty}d_{A_i}=\infty$, choose $A_p=\bigoplus_{j=1}^t C(Y_j)\otimes M_{p_j}$, such that 
\begin{equation}\label{equ_mindim}
    d_{A_p}\geq 3m_1+ 3m_2 +\dots+ 3m_n.
\end{equation}
Let
\[
k_0(\phi_{p,r}\otimes \text{id}_B)([v_1]\oplus [v_2]\oplus\dots\oplus [v_n])=[w_1]\oplus [w_2]\oplus\dots\oplus [w_t]\in \pi_0(\u(A_p\otimes B)).
\]
Since, $\phi_{p,r}\colon A_r\rightarrow A_p$ is diagonal, let 
    \[
\begin{pmatrix}
    (\lambda_{1}^{11}, \lambda_2^{11}, \dots, \lambda_{k_{11}}^{11}) & (\lambda_{1}^{12}, \lambda_2^{12}, \dots, \lambda_{k_{12}}^{12}) & \dots & (\lambda_{1}^{1t}, \lambda_2^{1t}, \dots, \lambda_{k_{1t}}^{1t}) \\[6pt]
    (\lambda_{1}^{21}, \lambda_2^{21}, \dots, \lambda_{k_{21}}^{21}) & ( \lambda_{1}^{22}, \lambda_2^{22}, \dots, \lambda_{k_{22}}^{22}) & \dots & (\lambda_{1}^{2t}, \lambda_2^{2t}, \dots, \lambda_{k_{2t}}^{2t}) \\[6pt]
    \vdots & \vdots & \ddots & \vdots \\[6pt]
    (\lambda_{1}^{n1}, \lambda_2^{n1}, \dots, \lambda_{k_{n1}}^{n1}) & (\lambda_{1}^{n2}, \lambda_2^{n2}, \dots, \lambda_{k_{n2}}^{n2}) & \dots & (\lambda_{1}^{nt}, \lambda_2^{nt}, \dots, \lambda_{k_{nt}}^{nt})
\end{pmatrix}
\]
be the matrix associated to $\phi_{p,r}$. 
Then,
\[
[w_j]=[\text{diag}(v_1\circ \lambda_1^{1j}, \dots, v_1\circ \lambda_{k_{1j}}^{1j}, v_2\circ \lambda_1^{2j}, \dots, v_2\circ \lambda_{k_{2j}}^{2j}, \dots, v_n\circ \lambda_1^{nj}, \dots, v_n\circ \lambda_{k_{nj}}^{nj})].
\]

We now write the set $\{k_{1j},k_{2j},\dots,k_{nj}\}$ as follows:
\[
\{k_{1j},k_{2j},\dots,k_{nj}\}=\{k_{n_1j},k_{n_2j},\dots,k_{n_sj}, k_{r_1j},k_{r_2j},\dots,k_{r_qj}, 0,0\dots,0\}
\]
such that
\begin{enumerate}
    \item  $k_{n_1j}=k_{n_2,j}=\dots =k_{n_sj}=1$,
    \item $k_{r_1j},k_{r_2j},\dots,k_{r_qj}\geq 2$.
\end{enumerate}

Then, by the definition of a diagonal *-homomorphism (Definition~\ref{def_diagonal_maps}) and~\eqref{equ_mindim}, we can see that 
\begin{equation}\label{dim_eqn}
p_j=\sum_{\beta=1}^s m_{n_\beta} +  \sum_{\alpha=1}^q m_{r_\alpha}k_{{r_\alpha}j}\geq 3m_1+ 3m_2+\dots+ 3m_n.
\end{equation}

By repeated use of Lemma~\ref{whiteheads}, we see that
\begin{equation}\label{whitehead_equn}
\begin{split}
w_j 
\sim_h &\;
\text{diag}\big(
    v_{n_1}\!\circ\!\lambda_1^{n_1j},
    v_{n_2}\!\circ\!\lambda_1^{n_2j},
    \dots,
    v_{n_s}\!\circ\!\lambda_1^{n_sj},\\
&\quad
    v_{r_1}\!\circ\!\lambda_1^{r_1j}, \dots, v_{r_1}\!\circ\!\lambda_{k_{r_1j}}^{r_1j},\,
    v_{r_2}\!\circ\!\lambda_1^{r_2j}, \dots, v_{r_2}\!\circ\!\lambda_{k_{r_2j}}^{r_2j},\, \dots,\,
    v_{r_q}\!\circ\!\lambda_1^{r_qj}, \dots, v_{r_q}\!\circ\!\lambda_{k_{r_qj}}^{r_qj}
\big)\\[4pt]
\sim_h &\;
\text{diag}\big(
    v_{n_1}\!\circ\!\lambda_1^{n_1j},
    v_{n_2}\!\circ\!\lambda_1^{n_2j},
    \dots,
    v_{n_s}\!\circ\!\lambda_1^{n_sj},\\
&\quad
    (v_{r_1}\!\circ\!\lambda_1^{r_1j}) \bullet (v_{r_1}\!\circ\!\lambda_2^{r_1j}) \bullet \dots \bullet (v_{r_1}\!\circ\!\lambda_{k_{r_1j}}^{r_1j}),
    0_{m_{r_1}},
    \underbrace{0_{m_{r_1}}, \dots, 0_{m_{r_1}}}_{k_{r_1j}-2\,\text{ times}},\\
&\quad
    (v_{r_2}\!\circ\!\lambda_1^{r_2j}) \bullet (v_{r_2}\!\circ\!\lambda_2^{r_2j}) \bullet \dots \bullet (v_{r_2}\!\circ\!\lambda_{k_{r_2j}}^{r_2j}),
    0_{m_{r_2}},
    \underbrace{0_{m_{r_2}}, \dots, 0_{m_{r_2}}}_{k_{r_2j}-2\,\text{ times}},\\
&\quad
    \dots,\,
    (v_{r_q}\!\circ\!\lambda_1^{r_qj}) \bullet (v_{r_q}\!\circ\!\lambda_2^{r_qj}) \bullet \dots \bullet (v_{r_q}\!\circ\!\lambda_{k_{r_qj}}^{r_qj}),
    0_{m_{r_q}},
    \underbrace{0_{m_{r_q}}, \dots, 0_{m_{r_q}}}_{k_{r_qj}-2\,\text{ times}}
\big).
\end{split}
\end{equation}
By~\eqref{homotopy_eqn}, for $1\leq \a\leq q$, and $1\leq \b\leq s$,
\[
\begin{pmatrix}
    v_{r_\a}&0\\
    0& 0_{m_{r_\a}}
\end{pmatrix}\sim_h 0\in \u_2(C(X_{r_\a}\otimes M_{m_{r_\a}}(B))),
\]
and 
\[
\begin{pmatrix}
    v_{n_\b}&0\\
    0& 0_{m_{n_\b}}
\end{pmatrix}\sim_h 0\in \u_2(C(X_{r_\b}\otimes M_{m_{r_\b}}(B))).
\]
Hence, by Lemma~\ref{homotopy_lem}, for $1\leq \a\leq q$, $1\leq y\leq k_{{r_\a}j}$, and  $1\leq \b\leq s$
\begin{equation}
\begin{pmatrix}
    v_{r_\a}\circ \lambda_{y}^{r_{\a}j}&0\\
    0& 0_{m_{r_\a}}
\end{pmatrix}\sim_h 0\in \u_2(C(Y_j\otimes M_{m_{r_\a}}(B))),
\end{equation}
and 
\begin{equation}\label{eqn_first_hom}
\begin{pmatrix}
    v_{n_\b}\circ \lambda_{1}^{n_{\b}j}&0\\
    0& 0_{m_{n_\b}}
\end{pmatrix}\sim_h 0\in \u_2(C(Y_j\otimes M_{m_{n_\b}}(B))).
\end{equation}
Hence, for $1\leq \a\leq q$
\begin{equation}\label{prod_hom_eqn}
\begin{split}
\begin{pmatrix}
    v_{r_\a}\circ \lambda_{1}^{r_{\a}j}&0\\
    0& 0_{m_{r_\a}}
\end{pmatrix}\bullet \begin{pmatrix}
    v_{r_\a}\circ \lambda_{2}^{r_{\a}j}&0\\
    0& 0_{m_{r_\a}}
\end{pmatrix}\bullet\dots\bullet \begin{pmatrix}
    v_{r_\a}\circ \lambda_{k_{r_\a}j}^{r_{\a}j}&0\\
    0& 0_{m_{r_\a}}
\end{pmatrix}&\sim_h 0 \in \u_2(C(Y_j\otimes M_{m_{r_\a}}(B))),\\[6pt]
\implies \begin{pmatrix}
    (v_{r_\a}\circ \lambda_{1}^{r_{\a}j})\bullet (v_{r_\a}\circ \lambda_{2}^{r_{\a}j})\bullet\dots\bullet (v_{r_\a}\circ \lambda_{k_{r_\a}j}^{r_{\a}j}) &0\\
    0& 0_{m_{r_\a}}
\end{pmatrix}&\sim_h 0\in \u_2(C(Y_j\otimes M_{m_{r_\a}}(B))).
\end{split}
\end{equation}
Moreover, since $k_{n_\beta j}=1$, for all $1\leq \beta\leq s$, by~\eqref{dim_eqn}
\begin{equation*}
\begin{split}
\sum_{\beta=1}^s m_{n_\beta}k_{{n_\beta}j} +  \sum_{\alpha=1}^q m_{r_\alpha}k_{{r_\alpha}j}&\geq \sum_{\beta=1}^s3m_{n_\beta}+ \sum_{\alpha=1}^q 3m_{r_\alpha}\\ \implies \sum_{\alpha=1}^q m_{r_\alpha}(k_{{r_\alpha}j}-2)\geq  \sum_{\beta=1}^s m_{n_\beta}(k{_{n_\beta j}}-2) + &\sum_{\alpha=1}^q m_{r_\alpha}(k_{{r_\alpha}j}-2)\geq \sum_{\beta=1}^s m_{n_\beta}+ \sum_{\alpha=1}^q m_{r_\alpha}\geq \sum_{\beta=1}^s m_{n_\beta}.
\end{split}
\end{equation*}
This implies that the last matrix homotopic to $w_j$ in ~\eqref{whitehead_equn} has enough space to accommodate the zero matrices, $0_{m_{n_1}}, 0_{m_{n_2}}, \dots 0_{m_{n_s}} $.  Hence by~\eqref{prod_hom_eqn} and~ \eqref{eqn_first_hom}
\begin{equation*}
\begin{split}
w_j 
\sim_h &\;
\text{diag}\big(
    v_{n_1}\!\circ\!\lambda_1^{n_1j},\,
    v_{n_2}\!\circ\!\lambda_1^{n_2j},\,\dots,\,
    v_{n_s}\!\circ\!\lambda_1^{n_sj},\\
&\quad (v_{r_1}\!\circ\!\lambda_1^{r_1j}) \bullet (v_{r_1}\!\circ\!\lambda_2^{r_1j}) \bullet \dots \bullet (v_{r_1}\!\circ\!\lambda_{k_{r_1j}}^{r_1j}),
    0_{m_{r_1}}, 
    \underbrace{0_{m_{r_1}}, \dots, 0_{m_{r_1}}}_{k_{r_1j}-2\text{ times}},\\
&\quad (v_{r_2}\!\circ\!\lambda_1^{r_2j}) \bullet (v_{r_2}\!\circ\!\lambda_2^{r_2j}) \bullet \dots \bullet (v_{r_2}\!\circ\!\lambda_{k_{r_2j}}^{r_2j}),
    0_{m_{r_2}}, 
    \underbrace{0_{m_{r_2}}, \dots, 0_{m_{r_2}}}_{k_{r_2j}-2\text{ times}},\\
&\quad \dots,\,
    (v_{r_q}\!\circ\!\lambda_1^{r_qj}) \bullet (v_{r_q}\!\circ\!\lambda_2^{r_qj}) \bullet \dots \bullet (v_{r_q}\!\circ\!\lambda_{k_{r_qj}}^{r_qj}),
    0_{m_{r_q}},
    \underbrace{0_{m_{r_q}}, \dots, 0_{m_{r_q}}}_{k_{r_qj}-2\text{ times}}
\big)\\[4pt]
\sim_h &\;
\text{diag}\big(
    v_{n_1}\!\circ\!\lambda_1^{n_1j}, 0_{m_{n_1}},\,
    v_{n_2}\!\circ\!\lambda_1^{n_2j}, 0_{m_{n_2}},\, \dots,\,
    v_{n_s}\!\circ\!\lambda_1^{n_sj}, 0_{m_{n_s}},\\
&\quad (v_{r_1}\!\circ\!\lambda_1^{r_1j}) \bullet (v_{r_1}\!\circ\!\lambda_2^{r_1j}) \bullet \dots \bullet (v_{r_1}\!\circ\!\lambda_{k_{r_1j}}^{r_1j}), 0_{m_{r_1}},\\
&\quad (v_{r_2}\!\circ\!\lambda_1^{r_2j}) \bullet (v_{r_2}\!\circ\!\lambda_2^{r_2j}) \bullet \dots \bullet (v_{r_2}\!\circ\!\lambda_{k_{r_2j}}^{r_2j}), 0_{m_{r_2}}, \dots,\\
&\quad (v_{r_q}\!\circ\!\lambda_1^{r_qj}) \bullet (v_{r_q}\!\circ\!\lambda_2^{r_qj}) \bullet \dots \bullet (v_{r_q}\!\circ\!\lambda_{k_{r_qj}}^{r_qj}), 0_{m_{r_q}}, 0_*
\big)\\[4pt]
\sim_h &\; 0 
\;\in\; \u(C(Y_j)\otimes M_{p_j}(B)),
\end{split}
\end{equation*}
where $*= \sum_{\a=1}^q m_{r_\a}(k_{{r_\a}j}-2)-\sum_{\b=1}^s m_{n_\b}$. 

Thus $[w_j]=0\in\pi_0(\u(C(Y_j)\otimes M_{p_j}(B)))$. Since $j$ is arbitrary, we get \[k_0(\phi_{p,r}\otimes \text{id}_B)[v]=0\in \pi_0(\u(A_p\otimes B)).\]

Let us now show surjectivity of $k_0(\iota_{2,A\otimes B})$. The idea here is to take an element in $\pi_0(\u_2(A\otimes B))$ which comes from an element at some finite stage. This element may not have a preimage, but we can push the element to a later stage so that its image there does have a preimage. Let us make this precise. 

Let $[a]\in \pi_0(\u_2(A\otimes B))$. Then there exists some $r\in\N$ and some $[v]\in \pi_0(\u_2(A_r\otimes B))$ such that \[k_0(\phi_{r,\infty}\otimes\text{id}_B \otimes \text{id}_{M_2})([v])=[a].\] Let $A_r=\bigoplus_{i=1}^n C(X_i)\otimes M_{m_i}$, and as done in the earlier part of the proof, write 
\[[v]=[v_1]\oplus [v_2]\oplus\hdots\oplus[v_{n}]\in\pi_0(\u_2(A_r\otimes B)),\]
where each $[v_i]\in \pi_0(\u_2(C(X_i)\otimes M_{m_i}(B)))$. Since $\lim\limits_{i\rightarrow\infty}d_{A_i}=\infty$, choose $A_p=\bigoplus_{j=1}^t C(Y_j)\otimes M_{p_j}$, such that \[d_{A_p}\geq 2m_1+2m_2+\hdots+2m_n.\]Furthermore, since $\phi_{p,r}\colon A_r\rightarrow A_p$ is diagonal, let 
 \[
\begin{pmatrix}
    (\lambda_{1}^{11}, \lambda_2^{11}, \dots, \lambda_{k_{11}}^{11}) & (\lambda_{1}^{12}, \lambda_2^{12}, \dots, \lambda_{k_{12}}^{12}) & \dots & (\lambda_{1}^{1t}, \lambda_2^{1t}, \dots, \lambda_{k_{1t}}^{1t}) \\[6pt]
    ( \lambda_{1}^{21}, \lambda_2^{21}, \dots, \lambda_{k_{21}}^{21}) & ( \lambda_{1}^{22}, \lambda_2^{22}, \dots, \lambda_{k_{22}}^{22}) & \dots & ( \lambda_{1}^{2t}, \lambda_2^{2t}, \dots, \lambda_{k_{2t}}^{2t}) \\[6pt]
    \vdots & \vdots & \ddots & \vdots \\[6pt]
    (\lambda_{1}^{n1}, \lambda_2^{n1}, \dots, \lambda_{k_{n1}}^{n1}) & (\lambda_{1}^{n2}, \lambda_2^{n2}, \dots, \lambda_{k_{n2}}^{n2}) & \dots & (\lambda_{1}^{nt}, \lambda_2^{nt}, \dots, \lambda_{k_{nt}}^{nt})
\end{pmatrix}
\]
be the matrix associated to $\phi_{p,r}$. Then, $\phi_{p,r}\otimes \text{id}_{M_2}$ is again diagonal, having the same associated matrix as that of $\phi_{p,r}$. Fix $1\leq j\leq t$, and write 
\[\{k_{1j},k_{2j},\dots,k_{nj}\}=\{k_{n_1j},k_{n_2j},\dots,k_{n_sj}, 0,0\dots,0\},\]
where $k_{n_\b j}\geq 1$, for all $1\leq \b\leq s$. For  $1\leq \b\leq s$, define
\begin{equation*}
\begin{split}
    u_{n_\b}^j:=v_{n_\b}\circ \lambda_1^{n_\b j}\bullet v_{n_\b}\circ \lambda_2^{n_\b j} \bullet\hdots\bullet v_{n_\b}\circ \lambda_{k_{n_\b j}}^{n_\b j} &\in \u_2(C(Y_j)\otimes M_{m_{n_\b}}(B)).
    \end{split}
\end{equation*}
Since $p_j\geq 2m_1+2m_2+\dots +2m_n$,
\begin{equation}
    w_j:=\text{diag}(u_{n_1}^j,u_{n_2}^j,\hdots,u_{n_s}^j, 0_{l_j})
\end{equation}
belongs to $\u(C(Y_j)\otimes M_{p_j}(B))$ and
 $[w_j]\in \pi_0(\u(C(Y_j)\otimes M_{p_j}(B)))$, where $l_j=p_j-\sum_{\b=1}^s 2m_{n_\b}$. Thus, we can define an element 
 \[[w]=[w_1]\oplus [w_2]\oplus\hdots\oplus [w_t]\in \pi_0(\u(A_p\otimes B)).\]
 
 Now, it suffices to show that \[k_0(\iota_{2,A_p\otimes B})([w])=k_0(\phi_{p,r}\otimes \text{id}_B \otimes \text{id}_{M_2})([v]).\] Let \[k_0(\iota_{2,A_p\otimes B})([w])=[x_1]\oplus [x_2]\oplus \hdots \oplus [x_t],\] where $[x_j]=\bigg[\begin{pmatrix}
    w_j &0\\
    0 &0_{p_j}
\end{pmatrix}\bigg]$, 
and \[k_0(\phi_{p,r}\otimes \text{id}_B\otimes \text{id}_{M_2})([v])=[y_1]\oplus [y_2]\oplus \hdots \oplus [y_t]\] where
\[
[y_j]=[\text{diag}(v_{n_1}\circ \lambda_1^{n_1j}, \dots, v_{n_1}\circ \lambda_{k_{n_1j}}^{n_1j}, v_{n_2}\circ \lambda_1^{n_2j}, \dots, v_{n_2}\circ \lambda_{k_{n_2j}}^{n_2j}, \dots, v_{n_s}\circ \lambda_1^{n_sj}, \dots, v_{n_s}\circ \lambda_{k_{n_sj}}^{n_sj})].
\]
Then again by Lemma~\ref{whiteheads}
\begin{equation}
y_j \sim_h 
\text{diag}\big(
    u_{n_1}^j, 
    \underbrace{0_{2m_{n_1}}, \hdots, 0_{2m_{n_1}}}_{k_{n_1j}-1\,\text{ times}},\;
    u_{n_2}^j, 
    \underbrace{0_{2m_{n_2}}, \hdots, 0_{2m_{n_2}}}_{k_{n_2j}-1\,\text{ times}},\;
    \hdots,\;
    u_{n_s}^j,
    \underbrace{0_{2m_{n_s}}, \hdots, 0_{2m_{n_s}}}_{k_{n_sj}-1\,\text{ times}}
\big).
\end{equation}
 Since $p_j= \sum_{\b=1}^s m_{n_\b}k_{n_\b j}$, we see that
\begin{equation}
\begin{split}
l_j + p_j 
&= 2p_j - \sum_{\beta=1}^s 2 m_{n_\beta} \\
&= (k_{n_1 j} - 1)\, 2 m_{n_1} + (k_{n_2 j} - 1)\, 2 m_{n_2} + \dots + (k_{n_s j} - 1)\, 2 m_{n_s}.
\end{split}
\end{equation}
   Hence, $[y_j]=[x_j]$ for $1\leq j\leq t$, showing that \[k_0(\iota_{2,A_p\otimes B})[w]=k_0(\phi_{p,r}\otimes \text{id}_B\otimes \text{id}_{M_2})[v].\] Finally, 
   \begin{equation}
\begin{split}
[a] 
&= k_0(\phi_{\infty,r} \otimes \mathrm{id}_B\otimes \text{id}_{M_2})([v]) \\
&= k_0(\phi_{\infty,p}\otimes \mathrm{id}_B\otimes \text{id}_{M_2}) \circ k_0(\phi_{p,r} \otimes \mathrm{id}_B\otimes \text{id}_{M_2})([v]) \\
&= k_0(\phi_{\infty,p} \otimes \mathrm{id}_B\otimes \text{id}_{M_2}) \circ k_0(\iota_{2, A_p \otimes B})([w]) \\
&= k_0(\iota_{2, A \otimes B}) \Big( k_0(\phi_{\infty,p} \otimes \mathrm{id}_B)([w]) \Big).
\end{split}
\end{equation}

    Thus $[a]$ has a preimage showing surjectivity of $k_0(\iota_{2, A\otimes B})$.
\end{proof}

Next, we show that the converse of Proposition~\ref{thm:growth-to-K} also holds true. 
For this, we first require the following lemma.
\begin{lem}\label{lem:homogeneous-quotient}
Let 
\[
A = \varinjlim (A_i, \varphi_i)
\] 
be a diagonal AH-algebra, and suppose 
\[
\lim_{i\to\infty} d_{A_i} < \infty.
\]
Then $A$ has a homogeneous quotient of the form $C(Y) \otimes M_L$, for some nonempty compact Hausdorff space $Y$ and some $L\in\mathbb{N}$.
\end{lem}

\begin{proof}
Here, we employ ideas similar to those used in \cite[section 5]{seth2023}. Since the connecting maps are injective, $d_{A_i}$ is an increasing sequence of natural numbers, and we may very well assume that there exists some $L\in\mathbb{N}$ such that $d_{A_i}=L$ for all $i\geq 1$. Let
\[
A_i = \bigoplus_{l=1}^{n_i} C(X_{i,l}) \otimes M_{m_{i,l}}.
\]  
Without loss of generality, we write \(A_i\) as
\[
A_i = A_i^{(L)} \oplus A_i^{(>L)},
\]  
where
\[
A_i^{(L)} = \bigoplus_{l \in I_i^{(L)}} C(X_{i,l}) \otimes M_{m_{i,l}} \quad\mbox{and}\quad
A_i^{(>L)} = \bigoplus_{l \in I_i^{(>L)}} C(X_{i,l}) \otimes M_{m_{i,l}},
\]  
with the index sets defined by
\[
I_i^{(L)} = \{ l \mid 1 \le l \le n_i, \; m_{i,l} = L \} \quad\mbox{and}\quad
I_i^{(>L)} = \{ l \mid 1 \le l \le n_i, \; m_{i,l} > L \}.
\]  
That is, \(A_i^{(L)}\) is the direct sum of all homogeneous summands in \(A_i\) with matrix size exactly \(L\), and \(A_i^{(>L)}\) is the direct sum of all summands with matrix size strictly greater than \(L\).
We claim that there exists an inductive sequence of C*-algebras 
 
 \[
C(X_{1,p_1}) \otimes M_L 
\xrightarrow{\psi_1} 
C(X_{2,p_2}) \otimes M_L 
\xrightarrow{\psi_2} 
C(X_{3,p_3}) \otimes M_L 
\longrightarrow \cdots 
\;\longrightarrow\; D,
\]
where for all $i\geq 1$, 
 $p_i\in I_i^{(L)}$, and each connecting map 
\[
\psi_i \colon C(X_{i,p_i}) \otimes M_L \longrightarrow C(X_{i+1,p_{i+1}}) \otimes M_L
\]
is defined by $\psi_i(f) = f \circ \delta_i$, for some continuous map $\delta_i : X_{i+1,p_{i+1}} \to X_{i,p_i}$.
 Moreover, these maps are compatible in the sense that
\begin{equation}\label{compatible}
\psi_i \circ \pi_{p_i} = \pi_{p_{i+1}} \circ \varphi_i,
\end{equation}
where $\pi_{p_i} \colon A_i \to C(X_{i,p_i})\otimes M_L$ denotes the canonical quotient map onto the $p_i$-th summand. It can then be shown that $D=C(X_\infty)\otimes M_L$, where $ X_\infty$  is a nonempty compact Hausdorff space given as the  inverse limit of spaces  $X_{i,p_i}$ and connecting maps $\delta_{i}$, defined precisely as
\[
X_\infty 
= 
\Bigl\{ (x_1,x_2,\dots) \in \prod_i X_{i,p_i} \;\colon\;
\delta_i(x_{i+1}) = x_i \text{ for all } i \Bigr\}.
\]
Since $\psi_i \,\circ\, \pi_{p_i} = \pi_{p_{i+1}}\, \circ \,\varphi_i$, we conclude that that $A$ has a nontrivial quotient, namely $C(X_\infty)\otimes M_L$.

 Finally, let us construct the required inductive limit sequence. For simplicity, write $I_i^{(L)}=\{1,2,\hdots, m_i\}$, where $1\leq m_i\leq n_i$.
     Suppose that for each $1\leq l\leq m_1$, there exists $A_{j_{l}}$, such that 
     \[
     \varphi_{j_{l},1}\left( C((X_{1,l})\otimes M_L)\right)\subset A_{j_{l}}^{(>L)}\quad\text{and}\quad \varphi_{j_{l},1}\left( C((X_{1,l})\otimes M_L)\right)\cap A_{j_{l}}^{(L)}=\emptyset.
     \]
 Let $M:=\max\{j_l\colon 1\leq l\leq m_1\}$. Since $d_{A_{M+1}}=L$, there exists a summand in $ A_{M+1}^{(L)}$, which is, say $C(X_{M+1, p_{M+1}})\otimes M_L$. By definition of diagonal *-homomorphism, there must exist a summand in $A_{M}^{(L)}$,  say $C(X_{M, {p_M}})\otimes M_L$, such that 
  \[
\prescript{p_M}{}{\varphi_M^{\,p_{M+1}}} \colon C(X_{M,p_M})\otimes M_L\rightarrow C(X_{M+1,p_{M+1}})\otimes M_L
 \]
 is nontrivial. Then, $\prescript{p_M}{}{\varphi_M^{\,p_{M+1}}}$ is a diagonal $*$-homomorphism and the tuple associated to it is the $({p_M}, p_{M+1})$-th entry of the matrix associated to $\varphi_{M}$. 
     Furthermore, this is the only nonzero entry in the ${p_{M+1}}$-th column of this matrix and has the form $(\lambda_{p_M})$ for some continuous map \[\lambda_{p_M} \colon X_{M+1,p_{M+1}}\rightarrow X_{M,p_M}.\] Proceeding similarly for $C(X_{M, p_M}) \otimes M_L$ and continuing inductively, we see that for each $1 \le i \le M$, there exists a summand in $A_i^{(L)}$, denoted by $C(X_{i, p_i}) \otimes M_L$, such that
\[
\prescript{p_i}{}{\varphi_i^{\,p_{i+1}}} \colon C(X_{i, p_i}) \otimes M_L \longrightarrow C(X_{i+1, p_{i+1}}) \otimes M_L
\]
is a nontrivial diagonal $*$-homomorphism. The tuple associated with it is the $(p_i, p_{i+1})$-th entry of the matrix associated to $\varphi_i$, which is the only nonzero entry in the $p_{i+1}$-th column and has the form $(\lambda_{p_i})$ for some continuous map 
\[
\lambda_{p_i} \colon X_{i+1, p_{i+1}} \longrightarrow X_{i, p_i}.
\]
This contradicts that $\varphi_{j_{p_1} ,1}(C(X_{1,p_1})\otimes M_L)\cap A_{j_{p_1}}^{(L)}=\emptyset$.
      Hence, there exists some $p=p_1\in I_1^{(L)}$, such that $\varphi_{i,1}\left(C(X_{1,p_1})\otimes M_L\right)\cap A_{i}^{(L)}\neq \emptyset$ for all $i\geq 2$. It is then easy to see the existence of the required inductive limit sequence starting at $C(X_{1,p_1})\otimes M_L$, of the form
      \[
      C(X_{1,p_1})\otimes M_L\xrightarrow{\prescript{p_1}{}{\varphi_1^{\,p_{2}}}} C(X_{2,p_2})\otimes M_L\xrightarrow{\prescript{p_2}{}{\varphi_2^{\,p_{3}}}}C(X_{3,p_3})\otimes M_L\longrightarrow\hdots\longrightarrow D
      \]
      provided the maps are compatible in the sense of ~\eqref{compatible}, which is what we show now. Recall that  \[\prescript{p_i}{}{\varphi_i^{\,p_{i+1}}}\colon C(X_{i, p_i}) \otimes M_L \longrightarrow C(X_{i+1, p_{i+1}}) \otimes M_L\] is a diagonal $*$-homomorphism and the tuple associated to it is the $(p_i,p_{i+1})$-th entry in the matrix representation of $\varphi_i$, which looks like 
\[
\begin{pmatrix}
* & \cdots & * & 0 & * & \cdots & * \\
* & \cdots & * & 0 & * & \cdots & * \\
\vdots & & \vdots & \vdots & \vdots & & \vdots \\
* & \cdots & * & 0 & * & \cdots & * \\
* & \cdots & * & (\lambda_i)& * & \cdots & * \\  % p_i-th row, p_{i+1}-th column
* & \cdots & * & 0 & * & \cdots & * \\
* & \cdots & * & 0 & * & \cdots & * \\
\vdots & & \vdots & \vdots & \vdots & & \vdots \\
* & \cdots & * & 0 & * & \cdots & *
\end{pmatrix}.
\]
See that, due to dimension restrictions, $(p_i, p_{i+1})$-th entry, denoted by $(\lambda_i)$ is the only nonzero entry in the $p_{i+1}$-th column of the above matrix. Finally, let $f:=(f_1^i,f_2^i, \hdots, f_{n_i}^i)\in A_i$, where $f^i_l\in C(X_{i,l})\otimes M_{m_{i,l}}$. Since $(\varphi_i(f))_{p_i+1}=f_{p_i}^i\circ \lambda_i$, we see
\begin{equation*}
\begin{split}
&\prescript{p_i}{}{\varphi_i^{\,p_{i+1}}} \circ \pi_{p_i} (f)
=\prescript{p_i}{}{\varphi_i^{\,p_{i+1}}}(f^i_{p_i})
=  f^i_{p_i} \circ \lambda_i,\, \text{and}\\
&  \pi_{p_{i+1}} \circ \varphi_i\, (f)= f^i_{p_i} \circ \lambda_i.
\end{split}
\end{equation*} 
\end{proof}

\begin{prop}[$K$-stability implies growth]\label{thm:K-to-growth}
Let 
\[
A = \varinjlim (A_i, \varphi_i)
\] 
be a diagonal AH-algebra. If \(A \otimes B\) is \(K\)-stable for every C*-algebra \(B\), then
\[
\lim_{i\to\infty} d_{A_i} = \infty.
\]
\end{prop}
\begin{proof}
Suppose that $A = \varinjlim\nolimits_{\,}(A_i, \varphi_i)$ is $K$-stable. To show that $\lim\limits_{i \to \infty} d_{A_i} = \infty$. Assume this is not the case; that is, $\varinjlim d_{A_i} < \infty$.   
 Let
\[
A_i = \bigoplus_{l=1}^{n_i} C(X_{i,l}) \otimes M_{m_{i,l}}.
\]  
    Then, proceeding as done in Lemma~\ref{lem:homogeneous-quotient}, we see that there exists some $L\in\mathbb{N}$, and an inductive sequence of C*-algebras 
  \[
C(X_{1,p_1}) \otimes M_L
\xrightarrow{\psi_1} 
C(X_{2,p_2}) \otimes M_L 
\xrightarrow{\psi_2} 
C(X_{3,p_3}) \otimes M_L 
\longrightarrow \cdots 
\;\longrightarrow\; D,
\]
where for all $i\geq 1$, 
 $1\leq p_i\leq n_i$, and each connecting map 
\[
\psi_i \colon C(X_{i,p_i}) \otimes M_L \longrightarrow C(X_{i+1,p_{i+1}}) \otimes M_L
\]
is defined by $\psi_i(f) = f \circ \delta_i$, for some continuous map $\delta_i : X_{i+1,p_{i+1}} \to X_{i,p_i}$.
 Moreover, these maps are compatible in the sense that
\begin{equation*}
\psi_i \circ \pi_{p_i} = \pi_{p_{i+1}} \circ \varphi_i,
\end{equation*}
where $\pi_{p_i} \colon A_i \to C(X_{i,p_i})\otimes M_L$ denotes the canonical quotient map onto the $p_i$-th summand. Then, $D=C(X_\infty)\otimes M_L$, where $ X_\infty$  is a nonempty compact Hausdorff space given as the inverse limit of spaces  $X_{i,p_i}$ and connecting maps $\delta_{i}$, is a nontrivial quotient of $A$. 

We now show that if such an inductive sequence of C*-algebras exists, then $A$ cannot be $K$-stable as follows:
For $i=1,2,\hdots$, denote $I_i=\text{Ker}(\pi_{p_i})$, and $I=  \varinjlim\nolimits_{\,}(I_i, \varphi_i|_{I_i}) $. Then we have a split short exact sequence 
\begin{equation}\label{ses}
0 \longrightarrow I_i \longrightarrow A_i \xrightarrow{\pi_{p_i}} C(X_{i,p_i}) \otimes M_L\longrightarrow 0.
\end{equation}
We shall now use tools from rational homotopy theory to arrive at a contradiction. For a C*-algebra $A$, and $m\geq 1$, denote \[f_m(A):= k_m(A)\otimes \mathbb{Q}.\] Then $f_m$ is also a continuous homology theory. Notice that, up to rationalization $\u_n(\mathbb{C})$ is well understood (see \cite[Example 1.6]{apurva}). In particular,
\[
f_m(M_n(\mathbb{C})) =
\begin{cases}
\mathbb{Q}, & \text{if } m = 1, 3, 5, \dots, 2n-1,\\[2mm]
0, & \text{otherwise.}
\end{cases}
\]
Choose $m$ odd such that $2L \leq m \leq 4L-1$. Then, $f_m(M_L(\mathbb{C})) = 0$ and $f_m(M_{2L}(\mathbb{C})) = \mathbb{Q}$, and from the commutative diagram 
\begin{center}
\begin{tikzcd}
0 \arrow{r} & f_m(C_*(X_\infty, M_L)) \arrow{r} \arrow{d} & f_m(C(X_\infty, M_L)) \arrow{r} \arrow{d}{f_m(\iota_{2, C(X_\infty,M_L)})} & 0 \arrow{r} \arrow{d}{f_m(\iota_{2, M_L})} & 0 \\
0 \arrow{r} & f_m(C_*(X_\infty, M_{2L})) \arrow{r} & f_m(C(X_\infty, M_{2L})) \arrow{r} & \mathbb{Q} \arrow{r} & 0,
\end{tikzcd}
\end{center}
we conclude that 
\[
f_m(\iota_{2, C(X_\infty,M_L)}) = k_m(\iota_{2, C(X_\infty,M_L)}) \otimes \mathrm{id}_\mathbb{Q}
\] 
cannot be surjective. Moreover,  for each $i$, the split short exact sequence in~\eqref{ses}
gives a short exact sequence
\[
0 \longrightarrow f_m(I_i) \longrightarrow f_m(A_i)\xrightarrow{f_m(\pi_{p_i})} f_m(C(X_{i,p_i}) \otimes M_L) \longrightarrow 0,
\]
which in turn gives a short exact sequence at the level of inductive limits (see \cite[Chapter III, Exercise 21]{lang2012algebra})
\[
0 \longrightarrow f_m(I) \longrightarrow f_m(A)\rightarrow f_m(C(X_\infty,  M_L) \longrightarrow 0.
\]
Hence, from the commutative diagram 
\begin{center}
\begin{tikzcd}
0 \arrow{r} & f_m(I) \arrow{r} \arrow{d} & f_m(A) \arrow{r} \arrow{d}{f_m(\iota_{2, A})} & f_m(C(X_\infty,M_L))\arrow{r} \arrow{d}{f_m(\iota_{2, C(X_\infty,M_L)})} & 0 \\
0 \arrow{r} & f_m(M_2(I)) \arrow{r} & f_m(M_2(A)) \arrow{r} & f_m(C(X_\infty, M_{2L}))\arrow{r} & 0,
\end{tikzcd}
\end{center}

we conclude that $f_m(\iota_{2, A})$ cannot be surjective, contradicting the  $K$-stability of $A$.  
\end{proof}

Combining Proposition~\ref{thm:growth-to-K} and Proposition~\ref{thm:K-to-growth}, we obtain the following characterization for $K$-stability of diagonal AH-algebras

\begin{thm}\label{mainthm}
    Let 
\[
A = \varinjlim (A_i, \varphi_i)
\] 
be a diagonal AH-algebra. Then the following statements are equivalent:
\begin{enumerate}
    \item $\lim\limits_{i \to \infty} d_{A_i} = \infty$.
    \item $A \otimes B$ is $K$-stable for any C*-algebra $B$.
\end{enumerate}
\end{thm}

 The above theorem allows us to show that many interesting classes of C*-algebras are $K$-stable. Second implication follows by combining Theorem~\ref{mainthm} with \cite[Theorem 5.7]{seth2023}.

 \begin{cor} Let $B$ be any C*-algebra. Then $A\otimes B$ is $K$-stable whenever $A$ is a 
     \begin{enumerate}
         \item  non-$\mathcal{Z}$-stable Villadsen algebra of the first kind as constructed in \cite{villadsen1998simple}, or
         \item unital $K$-stable AF-algebra.
     \end{enumerate}
 \end{cor}

 We end this paper with another consequence of Theorem~\ref{mainthm}. Recall that a unital C*-algebra has stable rank one if its set of invertible elements is dense. 
It was shown in \cite[Theorem 2.10]{rieffel12} that if a C*-algebra $A$ has stable rank one, then it is K$_1$-bijective; that is, for all $n \ge 2$, the map
\[
k_0(\iota_{n,A})\colon k_0(M_{n-1}(A))\rightarrow k_0(M_n(A))
\]
as defined in Section~\ref{sec_pre} is an isomorphism. It has been shown in Theorem~4.1 of \cite{elliott2009class}, that if $A = \varinjlim\nolimits_{\,}(A_i, \varphi_i)$ is a simple unital diagonal AH-algebra, then $A$ has stable rank one, consequently, such algebras are K$_1$-bijective. It is still not known if, for a simple unital diagonal AH-algebra, the stable rank one is preserved under tensoring with an arbitrary C*-algebra $B$, hence it is not possible to conclude K$_1$-bijectivity of tensor products of simple unital diagonal AH-algebras.  

 By Lemma~\ref{lem:homogeneous-quotient}, we can show that, under an additional assumption of infinite dimensionality, simple unital diagonal AH-algebras satisfy the growth condition in Theorem~\ref{thm:growth-to-K}, and hence their tensor product is not only K$_1$-bijective but in fact $K$-stable.  

 We state the following proposition, which is a straightforward implication of Lemma~\ref{lem:homogeneous-quotient}.

  \begin{prop}
Let 
\[
A = \varinjlim (A_i, \varphi_i)
\] 
be a simple, unital, infinite-dimensional diagonal AH-algebra. Then 
\[
\lim_{i \to \infty} d_{A_i} = \infty.
\]
\end{prop}

As a consequence of the above proposition, using Proposition~\ref{thm:growth-to-K}, we obtain the following result.
 \begin{cor}
      Let 
\[
A = \varinjlim (A_i, \varphi_i)
\] 
be a simple, unital, infinite-dimensional diagonal AH-algebra. Then $A\otimes B$ is $K$-stable for every C*-algebra $B$. 
 \end{cor}

\section*{References}
\printbibliography[heading = none]

@article{apurva,
  title={{AF}-algebras and rational homotopy theory},
  author={Seth, Apurva and Vaidyanathan, Prahlad},
  journal={New York Journal of Mathematics},
  volume={26},
  pages={931--949},
  year={2020}
}

@article{zhang,
  title={Matricial structure and homotopy type of simple {C*}-algebras with real rank zero},
  author={Zhang, Shuang},
  journal={Journal of Operator Theory},
  volume={26},
  number={2},
  pages={283--312},
  year={1991},
  publisher={JSTOR}
}

@article{thomsen,
  title={Nonstable {K}-Theory for operator algebras},
  author={Thomsen, Klaus},
  journal={K-theory},
  volume={4},
  number={3},
  pages={245--267},
  year={1991},
  publisher={Springer Netherlands}
}

@article{rieffel12,
  title={The homotopy groups of the unitary groups of non-commutative tori},
  author={Rieffel, Marc A},
  journal={Journal of Operator Theory},
  pages={237--254},
  volume={17},
  number={2},
  year={1987},
  publisher={JSTOR}
}

@article{vaidyanathan,
  title={{$K$}-stability of continuous {$C(X)$}-algebras},
  author={Seth, Apurva and Vaidyanathan, Prahlad},
  journal={Proceedings of the American Mathematical Society},
  volume={148},
  number={9},
  pages={3897--3909},
  year={2020}
}

@article{seth2023,
  title={{$K$}-stability of {A}{$\mathbb{T}$}-algebras},
  author={Seth, Apurva and Vaidyanathan, Prahlad},
  journal={Journal of Mathematical Analysis and Applications},
  volume={522},
  number={2},
  pages={126957},
  year={2023},
  publisher={Elsevier}
}

@article{zhang2001,
  title={On the homotopy type of the unitary group and the Grassmann space of purely infinite simple {C*}-algebras},
  author={Zhang, Shuang},
  journal={K-Theory},
  volume={24},
  number={3},
  pages={203--225},
  year={2001},
  publisher={Springer Science+ Business Media BV, Formerly Kluwer Academic Publishers BV}
}

@article{jiang1997,
  title={Nonstable {K}-theory for {$\mathcal{Z}$}-stable {C*}-algebras},
  author={Jiang, Xinhui},
  journal={arXiv preprint math/9707228},
  year={1997}
}

@article{hua2025k,
  title={$K$-stability of $\mathcal{Z}$-stable {C*}-algebras},
  author={Hua, Shanshan},
  journal={International Journal of Mathematics},
  volume={36},
  number={09},
  pages={2550022},
  year={2025},
  publisher={World Scientific}
}

@article{villadsen1998simple,
  title={Simple {C*}-algebras with perforation},
  author={Villadsen, Jesper},
  journal={Journal of Functional Analysis},
  volume={154},
  number={1},
  pages={110--116},
  year={1998},
  publisher={Academic Press}
}

@article{elliott2009class,
  title={A class of simple {C*}-algebras with stable rank one},
  author={Elliott, George A and Ho, Toan M and Toms, Andrew S},
  journal={Journal of Functional Analysis},
  volume={256},
  number={2},
  pages={307--322},
  year={2009},
  publisher={Elsevier}
}

@article {Andrew_classification,
    AUTHOR = {Toms, Andrew S.},
     TITLE = {On the classification problem for nuclear {C*}-algebras},
   JOURNAL = {Ann. of Math. (2)},
  FJOURNAL = {Annals of Mathematics. Second Series},
    VOLUME = {167},
      YEAR = {2008},
    NUMBER = {3},
     PAGES = {1029--1044},
}

@book{rordam2000introduction,
  title={An introduction to K-theory for {\textup{C*}}-algebras},
  author={R{\o}rdam, Mikael and Larsen, Flemming and Laustsen, Niels},
  volume={49},
  year={2000},
  publisher={Cambridge University Press}
}

@article {MR1916650,
    AUTHOR = {Villadsen, Jesper},
     TITLE = {Comparison of projections and unitary elements in simple
              {C*}-algebras},
   JOURNAL = {J. Reine Angew. Math.},
  FJOURNAL = {Journal f\"ur die Reine und Angewandte Mathematik. [Crelle's
              Journal]},
    VOLUME = {549},
      YEAR = {2002},
     PAGES = {23--45},
}

@article{handelman1978k,
  title={{$K_0$} of von Neumann and {AF} {C*}-algebras},
  author={Handelman, David},
  journal={The Quarterly Journal of Mathematics},
  volume={29},
  number={4},
  pages={427--441},
  year={1978},
  publisher={Oxford University Press}
}

@book{lang2012algebra,
  title={Algebra},
  author={Lang, Serge},
  volume={211},
  year={2012},
  publisher={Springer Science \& Business Media, New York, NY}
}
\end{document}